\documentclass[11pt]{amsart}
\usepackage[english]{babel}
\usepackage[utf8]{inputenc}
\usepackage[T1]{fontenc}
\usepackage{amsmath,amsfonts,amssymb,amsthm,amscd,array,
stmaryrd,mathrsfs,mathdots}
\usepackage[pagebackref]{hyperref}
\usepackage{mathabx}

\DeclareUnicodeCharacter{2264}{\ensuremath{\leqslant}}
\DeclareUnicodeCharacter{2265}{\ensuremath{\geqslant}}

\theoremstyle{plain}
\newtheorem{lem}{Lemma}[section]
\newtheorem{thm}[lem]{Theorem}
\newtheorem{cor}[lem]{Corollary}
\newtheorem{prop}[lem]{Proposition}

\theoremstyle{definition}
\newtheorem{defn}[lem]{Definition}
\newtheorem{fact}[lem]{Experimental Fact}
\newtheorem{rem}[lem]{Remark}
\newtheorem{ex}[lem]{Example}
\newtheorem{conj}[lem]{Conjecture}

\numberwithin{figure}{section}
\numberwithin{table}{section}

\let\ssection=\section
\renewcommand{\section}{\setcounter{equation}{0}\ssection}

\newcommand\notesn[1]{%
    \bgroup
    \renewcommand\thefootnote{\fnsymbol{footnote}}%
    \renewcommand\thempfootnote{\fnsymbol{mpfootnote}}%
    \footnotetext[0]{#1}%
    \egroup
}

\newcommand{\R}{\mathbb{R}}
\newcommand{\Z}{\mathbb{Z}}

\newcommand{\Q}{\mathbb{Q}}

\newcommand{\Id}{\mathrm{Id}}
\newcommand{\SL}{\mathrm{SL}}
\newcommand{\PSL}{\mathrm{PSL}}

\def\a{\alpha}
\def\b{\beta}
\def\d{\delta}
\def\D{\Delta}
\def\e{\varepsilon}

\def\si{\sigma}

\begin{document}

\title[$q$-deformed real numbers and continued fractions]
{Continued fractions for $q$-deformed real numbers, 
$\{-1,0,1\}$-Hankel determinants,\\ and Somos-Gale-Robinson sequences}
\notesn{Final version. To appear in Adv. in Appl. Math.}

\keywords{Hankel determinants, q-analogues, quadratic irrationals, continued fractions, Somos and Gale-Robinson sequences.}

\subjclass{Primary: 05A30, 11A55. Secondary: 11B37, 11C20, 37J70.}

\date{2024-09-19}

\author{Valentin Ovsienko}
\address{
Valentin Ovsienko,
Centre National de la Recherche Scientifique,
Laboratoire de Math\'ematiques de Reims, UMR~9008 CNRS et
Universit\'e de Reims Champagne-Ardenne,
U.F.R. Sciences Exactes et Naturelles,
Moulin de la Housse - BP 1039,
51687 Reims cedex 2,
France}
\email{valentin.ovsienko@univ-reims.fr}

\author{Emmanuel Pedon}
\address{
Emmanuel Pedon,
Laboratoire de Math\'ematiques de Reims, UMR~9008 CNRS et
Universit\'e de Reims Champagne-Ardenne,
U.F.R. Sciences Exactes et Naturelles,
Moulin de la Housse - BP 1039,
51687 Reims cedex 2,
France} 
\email{emmanuel.pedon@univ-reims.fr}

\begin{abstract}
$q$-deformed real numbers are power series with integer coefficients.
We study Stieltjes and Jacobi type continued fraction expansions of $q$-deformed real numbers
and find many new examples of such continued fractions.
We also investigate the corresponding sequences of Hankel determinants
and find an infinite family of power series for which several of the first  sequences of Hankel determinants consist of $-1,0$ and $1$ only.
These Hankel sequences satisfy Somos and Gale-Robinson recurrences.
\end{abstract}

\maketitle

\thispagestyle{empty}

\tableofcontents

\section{Introduction and main results} 

In this paper we study power series with integral coefficients
that represent so-called $q$-deformed real numbers, or ``$q$-reals'' for short.
We calculate the sequences of Hankel determinants of the simplest examples.
It turns out that (conjecturally) there exists a remarkable infinite family of $q$-reals that have a surprising property:
several first sequences of their Hankel determinants are periodic and consist of $-1,0$ and $1$ only.
This property is similar to that of the generating functions of Catalan and Motzkin numbers,
but the number of $\{-1,0,1\}$-sequences is greater.
Moreover, the obtained $\{-1,0,1\}$-Hankel sequences enjoy very special Somos and Gale-Robinson recurrences
related to integrable dynamics.

The initial definition of $q$-reals~\cite{SVRe} relies on continued fractions which are not of the
Stieltjes or Jacobi type.
One of our goals is to rewrite these series in the form of C-fractions and J-fractions
which is most close to the classical Stieltjes and Jacobi continued fractions.
In particular, we use the special class of Jacobi type continued fractions,
called H-fractions, that was introduced and studied in~\cite{Han}.
The problem remains open in general,
we solve it in particular cases.

\subsection{Continued fractions: classical examples} 

The following special class of continued fractions depending on one (formal) variable, that we denote~$q$,
are classically called {\it C-fractions}
\begin{equation}
\label{SJGen}
f(q) =
\cfrac{b_0}{1
          - \cfrac{b_1q^{p_1}}{1
          -\cfrac{b_2q^{p_2}}{\ddots}}}  
\end{equation}
Here $p=(p_1,p_2,\ldots)$ is a sequence of integers, $p_i\geq1$,
and $b=(b_0,b_1,b_2,\ldots)$ is a sequence of (real or complex)
coefficients, such that $b_i\neq0$ for all~$i\geq0$.
When $p_i\equiv1$, the C-fraction~\eqref{SJGen} is a classical Stieltjes continued fraction~\cite{Sti}, which
is also called a \textit{regular C-fraction}, or an \textit{S-fraction};
see~\cite{Fla,Vie} where combinatorial properties of continued fractions were studied.
Given a power series 
$$
f(q)=\sum_{i=0}^{\infty}f_iq^i,
$$ 
it can always be written as a C-fraction in a unique way~\cite{LeSc}.
For an explicit algorithm, see~\cite{Sok}.
C-fractions were extensively studied and applied to many sequences of integers; 
see, e.g.~\cite{Aig,Bar,Bus,Cig,Kra1,Kra2,LeSc,Sok} and references therein.

The generalized Jacobi continued fractions, or {\it J-fractions}, are of the form
\begin{equation}
\label{JacGen}
f(q) =
\cfrac{b_0}{1+qA_1(q)
          - \cfrac{b_1q^{p_1}}{1+qA_2(q)
          -\cfrac{b_2q^{p_2}}{\ddots}}}  ,
\end{equation}
where $A_i(q)$ are polynomials with $\deg(A_i)<p_i-1$.
When $p_i\equiv2$, \eqref{JacGen} is the classical Jacobi continued fraction,
or {\it regular J-fraction}.
As mentioned, every power series can be written in the form~\eqref{SJGen},
but a more flexible form~\eqref{JacGen} may be simpler and have more symmetries.
The counterpart is that such a fraction does not always exist for a given series.
Combinatorial theory of J-fractions was founded in~\cite{Fla,Vie}.
For a modern theory and applications, see~\cite{Kra1,Kra2,Han,HanE}.

The \textit{Catalan numbers} $C_n=1, 1, 2, 5, 14, 42,\ldots$ are related to the simplest example of
regular C-fraction.
Indeed, the generating function 
$$C(q)=\sum_{i=0}^\infty{}C_nq^n=\frac{1-\sqrt{1-4q}}{2q}$$
has the continued fraction expansion
$$
C(q) \quad=\quad
\cfrac{1}{1
          - \cfrac{q}{1
          -\cfrac{q}{\ddots}}}  
$$
see e.g.~\cite{AigEnum}.
Here and thereafter the {\it generating function} of a series is a function whose Taylor  expansion at $0$
coincides with the series.
Abusing the notation, we will often identify series, the corresponding generating functions, and continued fractions.
We will use the same notation for all of them.

Note also that the ``shifted series'' $\frac{C(q)-1}{q}$ has the following J-fraction expansion
$$
\frac{C(q)-1}{q}\quad=\quad
\cfrac{1}{1-2q
          - \cfrac{q^{2}}{1-2q
          -\cfrac{q^{2}}{\ddots}}}   
$$

The \textit{Motzkin numbers} $M_n=1, 1, 2, 4, 9, 21, 51,\ldots$ is another classical example.
The generating function 
$$
M(q)=\frac{1-q-\sqrt{(1+q)(1-3q)}}{2q^2}=\frac{1}{q}C\Big(\frac{q}{1+q}\Big)
$$ 
can be written as continued fractions
$$
M(q) \quad=\quad
\cfrac{1}{1-
           \cfrac{q}{1-
          \cfrac{q}{1-
          \cfrac{q^2}{ 1-
          \cfrac{q}{ 1-
          \cfrac{q}{1-
          \cfrac{q^2}{\ddots}}}}}}}
   \quad=\quad
\cfrac{1}{1-q
          - \cfrac{q^{2}}{1-q
          -\cfrac{q^{2}}{\ddots}}}         
          $$
Note that the second formula (that can be found in~\cite{BarSom}) is the simplest example of a regular J-fraction,
while the first one is a very nice example of C-fraction.

\subsection{$q$-reals: examples and ideas} 

The classical $q$-deformed integers are defined for $n\in\Z_{>0}$ by
\begin{equation}
\label{EGEq}
[n]_q:=1+q+q^2+\cdots+q^{n-1}=
\frac{1-q^n}{1-q}.
\end{equation}
They arose in the works of Euler and Gauss and play important role in combinatorics and mathematical physics.
The notion of $q$-deformed rationals~\cite{SVRa} and, more generally, that of $q$-deformed real numbers~\cite{SVRe}
extends $q$-integers.
Given $x\in\R$, the $q$-deformation $[x]_q$ is a series in~$q$ with integer coefficients.
A fundamental property of $q$-reals is $\PSL(2,\Z)$-invariance; see~\cite{LMG}.
Analytic properties of the series defining $q$-deformed real numbers were studied in~\cite{LMGOV}.

The simplest example of $q$-deformed irrational number is the $q$-deformation of~$\varphi=\frac{1+\sqrt{5}}{2}$,
called the ``golden ratio''.
This $q$-deformation will be denoted by $G(q)$ (instead of $\left[\varphi\right]_q$ as in~\cite{SVRe}).
It is given by the series
\begin{equation}
\label{GoEq}
\begin{array}{rcl}
G(q)&:=&
1 + q^2 - q^3 + 2 q^4 - 4 q^5 + 8 q^6 - 17 q^7 + 37 q^8 - 82 q^9 \\[4pt]
&& +\, 185 q^{10}- 423 q^{11} + 978 q^{12}-2283q^{13}+ 5373q^{14}-12735q^{15}\\[4pt]
&&+\,30372q^{16}-72832q^{17}+175502q^{18}-424748q^{19}+1032004q^{20} \cdots
\end{array}
\end{equation}
see~\cite{SVRe}.
 Note that the coefficients of this series are quite close to the sequence A004148; see~\cite{OEIS}.
The differences are the zero linear term in~\eqref{GoEq} and the alternating signs.
The series A004148 belongs to the class of sequences called in~\cite{BarSom} the 
``generalized Catalan numbers''\footnote{Note however that there are many other sequences known under this name.}.
It is connected to the Narayana triangle and has interesting combinatorial interpretations; see~\cite{BarN}. 

The generating function of the series~\eqref{GoEq} is
\begin{equation}
\label{GGF}
G(q)=
\frac{q^2+q-1+\sqrt{(1-q+q^2)(1+3q+q^2)}}{2q}.
\end{equation}
Similarly to the Motzkin case, the function~$G(q)$ 
can be obtained from the generating function of the Catalan numbers:
$$
G(q)=1+q-\frac{q}{1+q+q^2}\,C\bigg(\frac{q^2}{(1+q+q^2)^2}\bigg)
$$
(see~\cite{BarSom} for a similar expression).
The series~\eqref{GoEq} and its shifts have nice continued fraction expansions that have close resemblance 
with the Catalan and Motzkin continued fractions.

\begin{prop}
\label{TrueGoldProp}
The series~\eqref{GoEq} is represented by the following $2$-periodic C-fraction
\begin{equation}
\label{TrueGoldCF}
G(q)\;=\;
\cfrac{1}{1-
          \cfrac{q^2}{1+
          \cfrac{q}{1-
           \cfrac{q^2}{1+\cfrac{q}{\ddots}}}}} 
\end{equation}
\end{prop}

It turns out that the series~\eqref{GoEq} is also related to another, $3$-periodic, C-fraction.
Let us introduce the following notation for the shifted series
\begin{equation}
\label{NotGold}
G^{(1)}(q):=\frac{G(q)-1}{q},
\qquad
G^{(2)}(q):=\frac{G(q)-1}{q^2},
\qquad
G^{(3)}(q):=\frac{G(q)-1-q^2}{q^3},
\end{equation}
that will be useful in the sequel.
Note that the series $G^{(2)}(q)$ corresponds precisely to the sequence in the first column of A123634; see~\cite{OEIS}.

\begin{prop}
\label{GoldProp}
One has
\begin{equation}
\label{GoldJTF}
G^{(2)}(q)\quad=\quad
\cfrac{1}{1+
           \cfrac{q}{1+
          \cfrac{q}{1+
          \cfrac{q^3}{ 1+
          \cfrac{q}{ 1+
          \cfrac{q}{1+
          \cfrac{q^3}{\ddots}}}}}}}
\end{equation}
\end{prop}

The $q$-deformed golden ratio seems to be an inexhaustible source of beautiful continued fractions. Let us give two more formulas.

\begin{prop}
\label{GoldMotProp}
One has the following $2$-periodic continued J-fraction
\begin{equation}
\label{GoldMot}
G^{(2)}(q)
\quad=\quad
\cfrac{1}{1+q
          - \cfrac{q^2}{1+q
          +\cfrac{q^{3}}{1+q
          - \cfrac{q^2}{ 1+q
          + \cfrac{q^3}{ \ddots}}}}} 
\end{equation}
\end{prop}

Despite their simplicity, we did not find the C-fractions~\eqref{TrueGoldCF} and~\eqref{GoldJTF}
nor the J-fraction~\eqref{GoldMot} in the literature. 
Their proofs will readily follow from the functional equations \eqref{GREq} and~\eqref{GRBisEq}
for the generating functions $G(q)$ and $G^{(2)}(q)$. 

Alternatively, the series $G^{(2)}(q)$ can be represented by a $1$-periodic J-type continued fraction
which is a particular case of Theorem~\ref{FristThm}.

\begin{prop}
\label{GoldDvaProp}
One has
\begin{equation}
\label{GoldHF}
G^{(2)}(q)
\quad=\quad
\cfrac{1}{1+q-q^2
          + \cfrac{q^3}{1+q-q^2
          +\cfrac{q^{3}}{\ddots}}} 
\end{equation}
\end{prop}

Note that the connection between the continued fraction~\eqref{GoldHF}
and the sequence A004148 was observed in~\cite{BarN}.

We will  see in Section~\ref{XXX} that~\eqref{GoldHF}
belongs to the interesting class of {\it super $\d$-fractions} introduced and studied in~\cite{Han}.
When $\d=2$, they are called H-fractions and used to compute Hankel determinants.
They are particularly efficient in the case where some of the Hankel determinants associated with the series vanish.

We will also show that, unlike~\eqref{TrueGoldCF},~\eqref{GoldJTF}, and~\eqref{GoldMot}, 
the expression~\eqref{GoldHF} can be generalized to the infinite family of $q$-deformed  ``metallic numbers''. Recall that these numbers are the irrationals $y_n$ ($n\in\Z_{>0}$) defined by the following $1$-periodic regular continued fraction: 
\begin{equation}
\label{MetallY}
y_n=\left[n,n,n,\ldots\right]
=n+
\cfrac{1}{n
          + \cfrac{1}{n
          +\cfrac{1}{\ddots}}}
= \frac{n+\sqrt{n^2+4}}{2}
\end{equation}
In Theorem~\ref{FristThm}, we shall see that the $q$-deformation $[y_n]_q$ of any metallic number $y_n$ admits a $1$-periodic super $\d$-fraction.
For the golden ratio $y_1=\frac{1+\sqrt{5}}{2}=\varphi$, this is simply formula~\eqref{GoldHF} since $[y_1]_q=G(q)=1+q^2G^{(2)}(q)$. 

The next metallic number is $y_2=\sqrt{2}+1$ which is often called the ``silver ratio''.
The $q$-deformation $\left[y_2\right]_q$ is the series that starts as follows
\begin{equation}
\label{SilverSerEq}
\begin{array}{rcl}
S(q)
&=&
1+q+q^4-2q^6+q^7+4q^8-5q^9-7q^{10}+ 18q^{11}+ 7q^{12}-55q^{13}+ 18q^{14}\\[4pt]
&&+\, 146q^{15}- 155q^{16} - 322q^{17}+692q^{18}+ 476q^{19}- 2446q^{20}+ 307q^{21}\cdots
 \end{array}
\end{equation}
see~\cite{SVRe}.
This series was recently added to the OEIS; see Sequence~A337589.
Not much is known about it yet.
The generating function of the series~\eqref{SilverSerEq} is the function
\begin{equation}
\label{SilverGFEq}
S(q)=
\frac{q^3+2q-1+\sqrt{(1-q+q^2)(1+q+4q^2 +q^3 +q^4)}}{2q}.
\end{equation}

As in the case of the golden ratio, we will use the following notation for the shifted series
\begin{equation}
\label{NotSilver}
\begin{array}{rr}
\displaystyle
S^{(1)}(q):=\frac{S(q)-1}{q},
&
\displaystyle
S^{(2)}(q):=\frac{S(q)-1-q}{q^2},
\\[10pt]
\displaystyle
S^{(3)}(q):=\frac{S(q)-1-q}{q^3},
&
\displaystyle
S^{(4)}(q):=\frac{S(q)-1-q}{q^4}.
\end{array}
\end{equation}
Next statement is then the analogue of Proposition~\ref{GoldDvaProp} for the series~\eqref{SilverSerEq}. It is a particular case of Theorem~\ref{FristThm}.
\begin{prop}
\label{SilHFrac3}
We have
\begin{equation}
\label{RootTwoHF}
S^{(4)}(q)
\quad=\quad
\cfrac{1}{1+2q^2-q^3
          + \cfrac{q^{5}}{1+2q^2-q^3
          +\cfrac{q^{5}}{\ddots}}} 
\end{equation}
\end{prop}

More  continued fractions related to $S(q)$ will be presented in Section~\ref{HSilverSec} and other examples of metallic
$q$-numbers will be treated in Section~\ref{Metal3Sec}.

\subsection{Hankel determinants}

Given a power series $f(q)=\sum_{i=0}^{\infty}f_iq^i$
or simply a sequence of numbers $f=(f_i)_{i\in\Z_{\geq0}}$, the corresponding {\it Hankel determinants} are
the determinants of the $n\times{}n$ matrices
\begin{equation}
\label{HankDet}
\D^{(\ell)}_n(f)=\left|
\begin{array}{cccc}
{f}_\ell & {f}_{\ell+1} & \cdots & {f}_{\ell+n-1} \\[4pt]
{f}_{\ell+1} &{f}_{\ell+2} & \cdots & {f}_{\ell+n} \\[4pt]
\vdots & \vdots && \vdots \\
 {f}_{\ell+n-1} & {f}_{\ell+n} & \cdots & {f}_{\ell+2n-2}
\end{array}
\right|,
\end{equation}
where $\ell,n=0,1,2,3,\ldots$, and where $\D^{(\ell)}_0(f):=1$ by convention.
These determinants are important characteristics of $f$ and appear in a variety of subjects.
The number $\ell$ is called the ``shift'' of the determinant $\D^{(\ell)}_n(f)$.
When $\ell=0$, we use the notation $\D_n(f)$.

Let us recall the classical examples of the Catalan numbers, studied by Aigner in~\cite{Aig}.
The first two sequences of Hankel determinants in this case are identically equal to~$1$,
and the third Hankel sequence consists in all natural numbers:
$$
\D_n\left(C\right)=1,\,1,\,1,\ldots
\qquad
\D^{(1)}_n\left(C\right)=1,\,1,\,1,\ldots
\qquad
\D^{(2)}_n\left(C\right)=n+1.
$$
Note that the Hankel sequences $\D_n\left(C\right)$ and $\D^{(1)}_n\left(C\right)$ completely characterize 
the sequence of Catalan numbers.

Another classical example is that of the Motzkin numbers, also studied by Aigner in~\cite{AigM}.
In this case, the first sequence of Hankel determinants is still identically~$1$, while the second is $3$-antiperiodic 
(and thus $6$-periodic) and
consists of $-1,0$, and $1$:
$$
\D_n\left(M\right)=1,\,1,\,1,\ldots
\qquad\qquad
\D^{(1)}_n\left(M\right)=1,\;1,\,0,-1,-1,\,0,\ldots
$$
for $n=0,1,2,3,\ldots$
The third sequence of Hankel determinants of the Motzkin numbers is 
\begin{eqnarray*}
\D^{(2)}_n\left(M\right) &=& 1,\,2,\,2,\,3,\,4,\,4,\,5,\,6,\,6,\,7,\,8,\,8,\ldots
\end{eqnarray*}

It is classical that, whenever $f(q)$ can be written as a regular C-fraction or a regular J-fraction,
 the determinants~$\D_n(f)$ can be calculated explicitly; see, e.g.~\cite{Kra1,Kra2}.
Much work has been done to generalize these classical results to the cases of more general continued
fractions such as~\eqref{SJGen} and~\eqref{JacGen}; see, e.g.~\cite{Bus,Cig,Han,HanE,Sok}.
We will try to adopt these results in our situation, although this is not always straightforward.
Our proofs are based on the notion of H-fractions of~\cite{Han}.

Sequences with Hankel determinants
consisting of $0,1$ and $-1$ were considered in~\cite{Bar1,BarSom,BarN,WYZ};
according to~\cite{Bar1} the question of characterization of such sequences was asked by Michael Somos.
We contribute to this study.

In Section~\ref{MetalH} we will prove the following property of the $q$-deformed golden ratio~\eqref{GoEq}.
\begin{thm}
\label{HankelGold}
\begin{enumerate}
\item The first four forward shifted sequences of Hankel determinants
corresponding to the series~$G(q)$ of the $q$-golden number~\eqref{GoEq} are $4$-antiperiodic
$$
\D^{(\ell)}_{n+4}(G)=-\D^{(\ell)}_n(G),
\qquad
\ell=0,1,2,3,
$$
 (and thus $8$-periodic) and consist of $0,1$, and $-1$ only. The sequences start as follows
\begin{equation}
\label{HankelGoldEq}
\begin{array}{rclrrrrrrrr}
\D_n(G) &=& 1,&\,1,&\,1,&\;\;0,&-1,&-1,&-1,&\,0,&\ldots\\[4pt]
\D^{(1)}_n(G)&=& 1,&\,0,&-1,&\;1,&-1,&\,0,&\,1,&-1,&\ldots\\[4pt]
\D^{(2)}_n(G)&=& 1,&\,1,&\,1,&\,0,&-1,&-1,&-1,&\,0,&\ldots\\[4pt]
\D^{(3)}_n(G)&=& 1,&-1,&\,0,&\,0,&-1,&\,1,&\,0,&\,0,&\ldots
\end{array}
\end{equation}
for $n=0,1,2,3,\ldots$
\item The first three rows above are interconnected: for all $n$,
\begin{equation*}
\D_n(G)=(-1)^n \D_{n-2}^{(1)}(G)=\D_n^{(2)}(G).
\end{equation*}
\end{enumerate}
\end{thm}

The first formula in~\eqref{HankelGoldEq} is not completely new; see Sequence A123634.
Sequences with generating function satisfying quadratic functional equations similar to our main examples
were studied in~\cite{WYZ,BarSom,Bar,CH} and~\cite{Hon2}.

The sequences of Hankel determinants $\D^{(\ell)}_n(G)$ with $\ell\geq4$ 
do not consist solely of $0,\,1$ and~$-1$,
but they form interesting patterns.
For instance, 
\begin{eqnarray}
\label{Gold4Eq}
\D^{(4)}_n(G)&=&1,\,2,\,0,-2,-3,-4,\,0,\,4,\,5,\,6,\,0,-6,-7,-8,\,0,\,8,\ldots
\end{eqnarray}
Observe a similarity with the Hankel sequence of the Motzkin numbers $\D^{(2)}_n\left(M\right)$.

We will prove in Section~\ref{MetalH}
that the series $G(q)$ (see~\eqref{GoEq}) is completely characterized by the first four sequences of Hankel determinants.
More precisely, we have the following statement.
\begin{thm}
\label{RecoverThm}
The series $G(q)$ is the only series with Hankel determinants as in~\eqref{HankelGoldEq}.
\end{thm}
\noindent
Note that this kind of statements is obvious when the sequences of Hankel determinants contain no zero entries,
for instance the sequence of Catalan numbers is characterized by the first two sequences of Hankel determinants
consisting of $1$'s.
With presence of zeros the situation is more complicated.

For our second main example of silver ratio $S(q)$ (see~\eqref{SilverSerEq}), 
we have the following result that will be proved in Section~\ref{ProofSilerSec}.

\begin{thm}
\label{SilverHanThm}
\begin{enumerate}
\item The first four sequences of Hankel determinants corresponding to the series~$S(q)$ of the $q$-silver number are periodic with period~$12$
$$
\D^{(\ell)}_{n+12}(S)=\D^{(\ell)}_n(S),
\qquad
\ell=0,1,2,3,
$$
and consist in $-1,0,1$ only.
The sequences start as follows
\begin{equation}
\label{DelHan}
\begin{array}{rcrrrrrrrrrrrrr}
\D_n(S)
&=&
1,&\,1,&-1,&-1,&\,1,&\,0,&\,-1,&\,0,&\,0,&\,1,&\,0,&-1,&\ldots\\[4pt]
\D^{(1)}_n(S)
&=&
1,&\,1,&\,0,&-1,&\,0,&\,0,&-1,&\,0,&\,1,&\,1,&-1,&-1,&\ldots\\[4pt]
\D^{(2)}_n(S)
&=&
1,&\,0,&\,0,&-1,&\,0,&\,1,&-1,&-1,&\,1,&\,1,&-1,&\,0,&\ldots\\[4pt]
\D^{(3)}_n(S)
&=&
1,&\,0,&-1,&-1,&\,1,&\,1,&-1,&-1,&\,0,&\,1,&\,0,&\,0,&\ldots
\end{array}
\end{equation}
for $n=0,1,2,3,\ldots$
\item These four rows are interconnected: 
$$\Delta_n^{(\ell+1)}(S)=(-1)^{n-1}\Delta^{(\ell)}_{n+3}(S)
\quad \text{for }\ell=0,1,2.$$
\end{enumerate}
\end{thm}

It turns out that the next sequence is also $12$-periodic,
$\D^{(4)}_{n+12}(S)=\D^{(4)}_n(S)$,
with the following period
\begin{eqnarray*}
\D^{(4)}_n(S)
&=&
1,\,1,-2,-1,\,2,-1,-2,\,1,\,1,\,0,\,0,\,0,\,\ldots
\end{eqnarray*}
This differs the series $S(q)$ from $G(q)$; cf. formula~\eqref{Gold4Eq}.
The sequences $\D^{(\ell)}_n(S)$ with $\ell\geq5$ seem to be aperiodic.

Although Theorem~\ref{SilverHanThm} is quite similar to Theorem~\ref{HankelGold},
unlike the case of the golden ratio,
we are unable to recover the sequence~$S(q)$ from this information
and do not know if this is the only series with Hankel determinants as in~\eqref{DelHan}.

The number of  Hankel sequences consisting of $0,1$, and $-1$  for the $q$-deformation of 
the metallic numbers $y_k$ defined by~\eqref{MetallY} seems to grow as $k$ grows; see Conjecture~\ref{MainCon} below.

\subsection{Somos and Gale-Robinson recurrences}
It is very easy to check that
the first three sequences of Hankel determinants
$\D_n(G),\D^{(1)}_n(G)$, and $\D^{(2)}_n(G)$ (see~\eqref{HankelGoldEq}) satisfy the recurrence
\begin{equation}
\label{Somos4}
\D_{n+4}\D_n = \D_{n+3}\D_{n+1} - \D_{n+2}^2.
\end{equation}
Note also for comparison that the shifted Hankel sequence $\D^{(1)}_n\left(M\right)$ of the Motzkin numbers satisfies
the recurrence 
$
\D_{n+2}\D_n=\D_{n+1}^2-1.
$

Recurrence~\eqref{Somos4} is (an instance of) the {\it Somos}-4 sequence, whose general form is
$$
a_{n+4}a_n = \a\, a_{n+3}a_{n+1} - \b a_{n+2}^2,
$$
for arbitrary parameters~$\a,\b$.
This remarkable class of sequences was discovered by Michael Somos in the '80s.
It remained unnoticed for some time and was disclosed by David Gale~\cite{Gal,Gal1}.
Since then Somos sequences became very popular 
and were studied by many authors.

Relation between the Hankel determinants of sequences satisfying quadratic (``Catalan type'') recurrences and Somos-4 sequences
was discovered by Michael Somos who lectured about the subject in the fall of 2000
at  MIT Stanley Seminar in Combinatorics.
Somos conjectured, in general, the appearance of the Somos-4 sequence as the Hankel
determinants of a quadratic type sequence
(for the sequence A004148, see the Hankel determinant number wall A123634).
This relation was rediscovered by Paul Barry and also conjectured in~\cite{BarSom}.
The first detailed proof of this conjecture was given in~\cite{CH},
a deep connection to (hyper)elliptic curves and generalizations were obtained in~\cite{Hon2}.
Recurrence~\eqref{Somos4} can be considered as a part of these results.

Besides, recurrence~\eqref{Somos4} has an interesting geometric interpretation.
Taking independent variables~$(x,y,z,t)$, it  produces the map
$$
(x,y,z,t)\longmapsto\Big(y,z,t,\,\frac{yt-z^2}{x}\Big),
$$
which generates a $4$-dimensional discrete dynamical system {\it integrable} in the sense of Liouville-Arnold; 
see~\cite{Hon1, FH} and references therein.
The first three Hankel sequences 
$\D_n(G),\D^{(1)}_n(G)$, and $\D^{(2)}_n(G)$ thus
correspond to periodic trajectories of the above map, or equivalently to periodic solutions of~\eqref{Somos4}.

It is also easy to see that the first four sequences of Hankel determinants~\eqref{DelHan} satisfy the recurrence
\begin{equation}
\label{Somos6}
\D_{n+6}\D_n = \D_{n+5}\D_{n+1} - \D_{n+3}^2,
\end{equation}
which is a Somos-6 sequence.
This statement is closely related to the results of Andrew Hone~\cite{Hon2},
since the quadratic recurrence for the coefficients of~\eqref{SilverSerEq}
fits into the class considered in this paper; see Eq.~(4.22) in~\cite{Hon2}.
\medbreak

Let us conclude our introduction. 
A long series of computer experimentations lead us to the following conjecture.

\begin{conj}
\label{MainCon}
Let $k$ be a positive integer, and for $\ell\in\Z_{\geq 0}$ 
let $\Delta_n^{(\ell)}:=\Delta^{(\ell)}_{n}\bigl(\left[y_k\right]_q\bigr)$ denote as in~\eqref{HankDet} the $\ell$-shifted sequence of Hankel determinants associated with the $q$-deformation of the metallic number~$y_k$.

(a)
The $k+2$ sequences
$\Delta_n^{(0)},\Delta_n^{(1)},\dots,\Delta_n^{(k+1)}$ 
consist of~$-1,0,1$ only, and they are $2k(k+1)$-periodic when $k$ is even and $2k(k+1)$-antiperiodic (hence $4k(k+1)$-periodic) when $k$ is odd: 
\begin{equation*}
\D_{n+2k(k+1)}^{(\ell)}=(-1)^k\D_{n}^{(\ell)}\quad\text{for all } n\in\Z_{\geq 0}.
\end{equation*}
Moreover, all these sequences satisfy the recurrence
\begin{equation}
\label{GRk}
\D^{(\ell)}_{n+2k+2}\,\D^{(\ell)}_n = \D^{(\ell)}_{n+2k+1}\,\D^{(\ell)}_{n+1} - \bigl(\D^{(\ell)}_{n+k+1}\bigr)^2 \quad\text{for all } n\in\Z_{\geq 0}.
\end{equation}

(b)
The $k+1$ pairs of consecutive shifted sequences $(\Delta^{(\ell)}_{n},\Delta^{(\ell-1)}_{n})$ with $\ell=1,2,\ldots,k+1$ are interconnected by the formula:
\begin{equation*}
\Delta^{(\ell)}_{n}=(-1)^{n+\frac{k(k+2\ell+1)}{2}}\Delta^{(\ell-1)}_{n+k+1}
\quad\text{for all } n\in\Z_{\geq 0}.
\end{equation*}
\end{conj}

Recurrence~\eqref{GRk} is called the {\it three-term Gale-Robinson} recurrence; see~\cite{Gal1,FZ}.
Integrability of the Gale-Robinson systems was proved by Fordy and Hone~\cite{FH}.
The Hankel determinants of $q$-deformed metallic numbers are, conjecturally, periodic 
$\{-1,0,1\}$-solutions of the corresponding discrete integrable systems.

The status of the conjecture is as follows.
Theorems~\ref{HankelGold} and~\ref{SilverHanThm} confirm the conjecture for $k=1$ and $k=2$ (for arbitrary $n$).
Applying Algorithms of~\cite{Han} (Section~3), one can obtain computer assisted proofs for further small values of~$k$,
and this has been done for~$k\leq8$ and $n\leq300$.
More details for $k=3$ and~$4$, and the explicit sequences of the Hankel determinants in these cases 
will be given in Section~\ref{Metal3Sec}.

\medbreak

\noindent\textbf{Organization}.
The paper is organized as follows.

In Section~\ref{TotalRecall}, we present the definition of $q$-reals and give a few examples.
We follow~\cite{SVRe,SVRa} and emphasize the continued fraction presentation.
We also discuss the $\PSL(2,\Z)$-invariance property.
The family of $q$-deformed ``metallic'' irrationals is described with some details.

In Section~\ref{HFrSec}, we develop the technique based on continued fractions.
We use the notion of super $\d$-fractions and H-fraction introduced  in~\cite{Han}.
We consider the series of $q$-deformed metallic numbers and find their $\d$-fraction presentation.

In Section~\ref{HankSec}, we prove 
 Theorems~\ref{HankelGold},~\ref{RecoverThm}, and~\ref{SilverHanThm}.
We also calculate the first sequences of Hankel determinants
for several other examples of metallic numbers and show that the phenomenon
detected for~$G(q)$ and~$S(q)$ persists and is amplified.
We demonstrate appearance of the recurrence~\eqref{GRk}.
These results remain experimental for other metallic numbers.

The final Section~\ref{EverySec} contains miscellaneous examples of continued fractions
representing $q$-numbers.
The general theory of C- and J-fraction presentation of $q$-numbers is yet to be developed.

\section{$q$-deformed real numbers} \label{TotalRecall}

In this section, we present the definition of $q$-deformed real numbers~\cite{SVRa,SVRe}
and give several examples.
The main property of $q$-reals is that of $\PSL(2,\Z)$-invariance~\cite{LMG}.
Other important properties have recently been studied in~\cite{Bap,Ove,LMGOV,SVS,OR,Ren,Sik}.

\subsection{$q$-deformed continued fractions}
The original definition consists in $q$-deformation of continued fractions,
but it is not of the form~\eqref{SJGen}.

Given $x\in\R$, let us consider its standard continued fraction expansions
$$
x\;=\;
c_1 - \cfrac{1}{c_2 
          - \cfrac{1}{\ddots } }
\qquad\hbox{and}\qquad
x\;=\;
a_1 + \cfrac{1}{a_2 
          + \cfrac{1}{\ddots } },
$$
where $(c_1,c_2,\ldots)$ and $(a_1,a_2,\ldots)$ are sequences of integers such that
$a_i\geq1$ and $c_i\geq2$, for all $i\geq2$.
The first continued fraction expansion is known under the name of Hirzebruch-Jung continued fraction
(sometimes called the ``negative'', ``minus'', or ``reversal'' continued fraction),
the second one is the most classical  continued fraction expansion.
We will use the notation
$$
x=\llbracket{}c_1,c_2\ldots\rrbracket{}
\qquad\hbox{and}\qquad
x=\left[a_1,a_2,\ldots\right],
$$
respectively.
Note that the first notation is due to Hirzebruch.
The coefficients $a_i$ and $c_j$ of the above expansions
are connected by the Hirzebruch
formula; see, e.g.~\cite{Hir,MGO}.

The above continued fractions are finite if and only if $x$ is rational
and infinite (converging to $x$) when $x$ is irrational.

\begin{defn} [\cite{SVRa,SVRe}]
\label{MainDef}
The $q$-analogue of~$x$ is the formal series $[x]_q$ defined by any of the following (equal) continued fractions:
\begin{eqnarray}
\label{QxEq1}
[x]_q
&=&
[c_1]_{q} - \cfrac{q^{c_{1}-1}}{[c_2]_{q} 
          - \cfrac{q^{c_{2}-1}}{\ddots } }\\[10pt]
&=&
[a_1]_{q} + \cfrac{q^{a_{1}}}{[a_2]_{q^{-1}} 
          + \cfrac{q^{-a_{2}}}{[a_{3}]_{q} 
          +\cfrac{q^{a_{3}}}{[a_{4}]_{q^{-1}}
          + \cfrac{q^{-a_{4}}}{
        \ddots} }}}
        \label{QxEq2}
 \end{eqnarray}
where $[n]_q$ stands for the $q$-integer as in~\eqref{EGEq}, and $[n]_{q^{-1}}=q^{1-n}[n]_q$ is the same expression with reciprocal parameter.
\end{defn}

The continued fractions~\eqref{QxEq1} and~\eqref{QxEq2} coincide when $x$ is rational~\cite{SVRa}.
When $x$ is irrational, the convergence of the expansions~\eqref{QxEq1} and~\eqref{QxEq2} and their coincidence is
guaranteed by the stabilization phenomenon highlighted in~\cite{SVRe}.
Note that, in the rational case, the second formula~\eqref{QxEq2} 
was also suggested (independently and almost simultaneously) in~\cite{Sik}.
Let us also mention that, when $x$ is rational, $[x]_q$ is a rational function in~$q$ that has many nice properties, 
such as unimodality and total positivity; see~\cite{SVRa,OR,CSS,Ove}. Another important property is that the $q$-deformation of any quadratic irrational number is necessarily periodic; see~\cite{LMG}. This will be observed in all the examples treated in our article.

\begin{rem}
The continued fraction~\eqref{QxEq1} has the type of \textit{Thron  fraction}, or \textit{T-fraction}; see~\cite{Thr}.
Indeed, the degree of every numerator is equal to the degree of the preceding polynomial.
Continued fractions of this type were used in various aspects of the theory of continued fractions and combinatorics,
but they do not allow to calculate Hankel determinants.
\end{rem}

\begin{ex}
(a)
The simplest example of a $q$-deformed irrational
is the golden ratio $\varphi=\frac{1+\sqrt{5}}{2}$ already discussed in the introduction.
The regular continued fraction of the golden ratio is 
$\varphi=[1,1,1,\ldots]$, and the Hirzebruch-Jung expansion is 
$\varphi=\llbracket{}2,3,3,3,\ldots\rrbracket{}$.
Recall that the $q$-deformation of $\varphi$ was denoted by~$G(q)$.
Formulas~\eqref{QxEq1} and~\eqref{QxEq2} then read
\begin{equation}
\label{GoldNaive}
\begin{array}{rcl}
G(q)
&=&
1+q- \cfrac{q}{1+q+q^2
          - \cfrac{q^{2}}{1+q+q^2
          -\cfrac{q^{2}}{\ddots
       }}} \\[50pt]
&=&
1 + \cfrac{q^{2}}{q
          + \cfrac{1}{1 
          +\cfrac{q^{2}}{q
          + \cfrac{1}{\ddots
       }}}} 
  \end{array}
\end{equation}
The explicit power series is as in~\eqref{GoEq}.

(b)
The regular continued fraction of the silver ratio $y_2$ is $1$-periodic:
$y_2=[2,2,2,\ldots]$, and the Hirzebruch-Jung expansion has period~$2$:
$y_2=\llbracket{}3,2,4,2,4,2,4,\ldots\rrbracket{}$.
The $q$-deformation $S(q)$ is then given by the series represented by the following $2$-periodic continued fractions
\begin{equation}
\label{SilverEq}
\begin{array}{rcl}
S(q)
&=&
[3]_q- \cfrac{q^{2}}{1+q
          - \cfrac{q}{1+q+q^2+q^3
          -\cfrac{q^{3}}{1+q
          - \cfrac{q}{\ddots
       }}}} \\[66pt]
&=&
1 +q+ \cfrac{q^{4}}{q+q^2
          + \cfrac{1}{1 +q
          +\cfrac{q^{4}}{q+q^2
          + \cfrac{1}{\ddots
       }}}} 
 \end{array}
\end{equation}

(c)
The formula for the power series $\left[\sqrt{2}\right]_q$ can be deduced from~\eqref{SilverEq} with the help of the recurrence formula~\eqref{RecEq} below, since $\left[\sqrt{2}\right]_q=S^{(1)}(q)$.

(d) For more examples, see~\cite{SVRe,LMG}.
\end{ex}

Clearly, the continued fractions~\eqref{GoldNaive},~\eqref{SilverEq},
and more generally \eqref{QxEq1} and~\eqref{QxEq2}, are neither C-fractions nor J-fractions,
since the powers of~$q$ in the numerators are not high enough.
Rewriting one of the fractions~\eqref{QxEq1} or~\eqref{QxEq2} in the form~\eqref{SJGen}  
and/or~\eqref{JacGen} is a challenging open problem.

\subsection{$\PSL(2,\Z)$-invariance}\label{SLSec}
Modular invariance, i.e., $\PSL(2,\Z)$-invariance is the main property of $q$-numbers.
It can be used as an equivalent (and perhaps more conceptual) definition, comparing to Definition~\ref{MainDef}.

The group $\SL(2,\Z)$ is the group of unimodular matrices with integer coefficients
$$
M=\begin{pmatrix}r&v\\s&u\end{pmatrix},
\qquad 
r,v,s,u\in\Z,
\quad
ru-vs=1.
$$
It acts on $\R\cup\{\infty\}$ by linear-fractional transformations: 
\begin{equation}
\label{LFAct}
M\cdot x=\frac{rx+v}{sx+u}.
\end{equation}
This action is effective for the modular group $\PSL(2,\Z)$, which is
the quotient of $\SL(2,\Z)$ by its center $\{\pm\Id\}$.
It can be generated by two elements, and the standard choice of generators is
$$
T=\begin{pmatrix}
1&1\\[4pt]
0&1
\end{pmatrix},
\quad
\quad
S=\begin{pmatrix}
0&-1\\[4pt]
1&0
\end{pmatrix}
$$
with the relations $S^2=(TS)^3=\Id$.

As in~\cite{SVRa,LMG}, consider the following matrices depending on~$q$
\begin{equation}
\label{TqSq}
T_q=\begin{pmatrix}
q&1\\[4pt]
0&1
\end{pmatrix},
\quad
\quad
S_q=\begin{pmatrix}
0&-1\\[4pt]
q&0
\end{pmatrix}.
\end{equation}
Viewed as elements of $\mathrm{PGL}(2,\Z[q,q^{-1}])$, 
the matrices $T_q$ and $S_q$ satisfy the same relations as $T$ and $S$,
namely $S_q^2=(T_qS_q)^3=\Id$.
Therefore, they generate a representation 
$$
\rho:\PSL(2,\Z)\to
\mathrm{PGL}(2,\Z[q,q^{-1}]),
$$ 
and hence an action on the space $\Z((q))\cup\{\infty\}$ of formal Laurent series in~$q$
defined by linear-fractional transformations as in formula~\eqref{LFAct}.
$\PSL(2,\Z)$-invariance then reads
$$
\left[M\cdot{}x\right]_q=\rho(M)\cdot{}\left[x\right]_q.
$$
Since the linear-fractional action \eqref{LFAct} is transitive on~$\Q\cup\{\infty\}$,
one can understand $q$-rationals as an orbit of one point, 
that can be chosen $[0]_q=0$, or $[1]_q=1$, etc. under the $\PSL(2,\Z)$-action on $\Z((q))\cup\{\infty\}$ defined by~\eqref{TqSq}.

Using the generators~\eqref{TqSq}, this can be stated as recurrence relations
\begin{equation}
\label{RecEq}
\left[x+1\right]_q=q\left[x\right]_q+1
\qquad\qquad
\left[-\frac{1}{x}\right]_q=-\frac{1}{q\left[x\right]_q},
\end{equation}
and when $x\in\Q$, these relations suffice to determine the rational function~$\left[x\right]_q$
if we know that $\left[0\right]_q=0$.
Formulas~\eqref{QxEq1} and~\eqref{QxEq2} readily follow from the $\PSL(2,\Z)$-invariance, at least for rational~$x$.
Indeed, if $x=\llbracket{}c_1,c_2\ldots,c_n\rrbracket{}$ then
$$
\left[x\right]_q=T_q^{c_1}S_qT_q^{c_2}S_q\cdots{}T_q^{c_n}\cdot0,
$$
hence~\eqref{QxEq1}.
The second formula~\eqref{QxEq2} involving the regular continued fraction can be deduced in a similar way; see~\cite{LMG}.

A straightforward generalization of the first formula in~\eqref{RecEq} will be of interest:
\begin{equation}
\label{qtrans}
[x+k]_q=q^k[x]_q+[k]_q\quad (x\in\R,\ k\in\Z_{\geq0}).
\end{equation} 

Let us also mention that the above $\PSL(2,\Z)$-invariance can  be understood 
as invariance with respect to the Burau representation of
the braid group $B_3$; see~\cite{Bap,SVS}.

\subsection{The gap theorem}
We give a refined form of the ``gap theorem'' which was proved in~\cite{SVRe}.

\begin{thm}
\label{GapThm}
If $k\leq{}x<k+\frac{1}{n}$ for some $k\in\Z_{\geq 0}$ and $n\in\Z_{> 0}$, then the series $\left[x\right]_q$ starts as follows
\begin{equation*}
\left[x\right]_q
=[k]_q\,+\,q^{k+n}+\kappa_{k+n+1}\,q^{k+n+1}+\cdots
\end{equation*}
i.e., if $k\geq 1$,
\begin{equation*}
\left[x\right]_q
=1+q+\cdots+q^{k-1}\,+\,q^{k+n}+\kappa_{k+n+1}\,q^{k+n+1}+\cdots
\end{equation*}
where $\kappa_i$ are some integer coefficients.
\end{thm}

\begin{proof}
The case $n=1$ was proved in~\cite{SVRe} (see Theorem~2).
Using~\eqref{qtrans}, this implies that, if $-n-1\leq{}x<-n$, then~$\left[x\right]_q$ is of the form
$$
\left[x\right]_q=-q^{-n-1}\left(1+\alpha_{1}\,q+\a_{2}\,q^{2}+\cdots\right)
$$
with $\a_{i}$ integers.
Applying the second equation in~\eqref{RecEq}, we conclude that
if $0\leq{}x<\frac{1}{n}$, then 
\begin{equation*}
\left[x\right]_q=q^n\left(1+\b_{1}\,q+\b_{2}\,q^{2}+\cdots\right)
\end{equation*}
with $\b_{i}$ integers.
The result follows applying again~\eqref{qtrans}.
\end{proof}

One can observe the gaps of length $1$ and $2$ in the examples~\eqref{GoEq} and
the series~$[\d]_q$ after~\eqref{SilverEq}, respectively.

The gap theorem will be important for the sequel. In particular, it implies that the expression
\begin{equation}\label{defSiq}
\si_q(x):=\frac{\left[x-k\right]_q}{q}=\frac{\left[x\right]_q-[k]_q}{q^{k+1}}
\quad(\text{when }k≤x<k+1)
\end{equation}
is  still a power series.
It turns out that these ``shifted series'' have better continued fraction expansions; formulas \eqref{GoldProp} and \eqref{GoldDvaProp} 
give two such examples, others will be obtained in Section~\ref{XXX}. 

\subsection{The metallic irrationals}

Our main examples of $q$-reals are $q$-deformations of the ``metallic numbers''~$y_k$ defined in~\eqref{MetallY}.

\begin{prop}
\label{MetRelProp}
For any integer $n$, let $y_n=\frac{n+\sqrt{n^2+4}}{2}=[n,n,n,\ldots]$ be the $n$-th metallic number.
Its $q$-deformation $[y_n]_q$ is characterized by the following functional equation
\begin{equation}
\label{RelY}
q\left[y_n\right]_q^2+\left((1+q^n)(1-q) - q\left[n\right]_q\right)\left[y_n\right]_q=1.
\end{equation}
\end{prop}

\begin{proof}
By definition \eqref{QxEq2}, the $q$-deformation~$[y_n]_q$ of the regular continued fraction $y_n=[n,n,n,\ldots]$ is characterized by the functional equation
\begin{eqnarray*}
[y_n]_q &=&[n]_q+\cfrac{q^n}{[n]_{q^{-1}}+\cfrac{q^{-n}}{[y_n]_q}},
\end{eqnarray*}
and therefore
\begin{eqnarray*}
[y_n]_q&=&
\frac{[n]_q[y_n]_q[n]_{q^{-1}}+q^{-n}[n]_q+q^n[y_n]_q}{[y_n]_q[n]_{q^{-1}}+q^{-n}}.
\end{eqnarray*}
Since $[n]_{q^{-1}}=q^{1-n}[n]_q$, after some simplification we obtain~\eqref{RelY}.
\end{proof}

\begin{cor} The $n$-th $q$-metallic number has the explicit expression
\begin{equation}\label{ExprY}
[y_n]_q=\frac{1}{2q}\left(q[n]_q+(q^n+1)(q-1)+\sqrt{\left(q[n]_q+(q^n+1)(q-1)\right)^2+4q}\right)
\end{equation}
and its power series expansion around $q=0$ is of the form
\begin{equation}\label{GapY}
[y_n]_q=1+q+\cdots+q^{n-1}+q^{2n}+\sum_{i=2n+1}^{\infty}\kappa_i q^i
\end{equation}
with coefficients $\kappa_i\in\Z_{≥0}$ (notice the gap of length $n$ in the powers of $q$).
\end{cor}

\begin{proof}
The formula for $[y_n]_q$ (appearing first in~\cite{LMG}) is an immediate consequence of~\eqref{RelY}, while the series expansion follows from the gap theorem (Theorem~\ref{GapThm}) since one has $n\leq y_n<n+\frac{1}{n}$ for all $n$.
\end{proof}

\begin{ex}
\label{MainEx}
(a)
The first metallic number is the golden ratio $\varphi=y_1$, already discussed in the introduction.
The generating function $G(q)=[y_1]_q$ of the $q$-deformed golden ratio
is given by~\eqref{GGF} and expands as the power series \eqref{GoEq}.
It satisfies (and is characterized by) the  functional equation
\begin{equation}
\label{GREq}
q\,G(q)^2+
\left(1-q-q^2 \right)G(q) =1,
\end{equation}
(see also~\cite{SVRe}).
Recalling that $G^{(1)}(q)$, $G^{(2)}(q)$, and $G^{(3)}(q)$ denote the generating functions of the shifted series~\eqref{NotGold},
we also have
\begin{equation}
\label{GRBisEq}
\begin{array}{rcl}
q^2G^{(1)}(q)^2 + \left(1+q-q^2 \right)G^{(1)}(q) &=&q,\\[4pt]
q^3G^{(2)}(q)^2 + \left(1+q-q^2 \right)G^{(2)}(q) &=&1,\\[4pt]
q^4 G^{(3)}(q)^2+\left(1+q-q^2+2q^3\right)G^{(3)}(q)&=&-1+q-q^2.
\end{array}
\end{equation}

(b)
The second example is the silver ratio~$y_2$.
The generating function~$S(q)=[y_2]_q$ of its $q$-deformation is as in~\eqref{SilverGFEq}, it expands into the power series \eqref{SilverSerEq} and it satisfies the following functional equation
\begin{equation}
\label{SREq}
q\,S(q)^2+
\left(1-2q-q^3\right)S(q) =1.
\end{equation}
Let~$S^{(1)}(q),S^{(2)}(q),S^{(3)}(q)$, and~$S^{(4)}(q)$ be the generating functions of the shifted series~\eqref{NotSilver}.
It easily follows from~\eqref{SREq} that these functions
satisfy
\begin{equation}
\label{ShSRFEq}
\begin{array}{rcl}
q^2\,S^{(1)}(q)^2+ \left(1-q^3\right)S^{(1)}(q) &=& 1+q^2,\\[4pt]
q^3\,S^{(2)}(q)^2+ \left(1+2q^2-q^3\right)S^{(2)}(q) &=& q^2,\\[4pt]
q^4\,S^{(3)}(q)^2+ \left(1+2q^2-q^3\right)S^{(3)}(q) &=& q,\\[4pt]
q^5\,S^{(4)}(q)^2+ \left(1+2q^2-q^3\right)S^{(4)}(q) &=& 1.
\end{array}
\end{equation}

(c) The third example  $y_3=\frac{3+\sqrt{13}}{2}$ is sometimes called the ``bronze ratio''.
The generating function $B(q)$ of~$\left[y_3\right]_q$ is given by 
\begin{equation}
\label{BronzeGF}
B(q)=\frac{q^4+q^2+2q-1+\sqrt{(1-q+q^2)(1+q+2q^2+5q^3+2q^4+q^5+q^6)}}{2q}
\end{equation}
and the series starts as follows
\begin{eqnarray*}
B(q) &=&
1+q+q^2+q^6-q^8-2q^9+2q^{10}+4q^{11}+q^{12}-11q^{13}-7q^{14}+ 15q^{15}\\
&&+\, 34q^{16}- 17q^{17}- 83q^{18}- 38q^{19}+ 189q^{20}+ 215q^{21}- 260q^{22}+ \cdots
\end{eqnarray*}
(This sequence is not in the OEIS.) Moreover, $B(q)$ is characterized by the functional equation
$$
q\,B(q)^2+\left(1-2q-q^2-q^4\right)B(q)=1.
$$
As concerns the shifted series $[y_3-1]_q=\left[\frac{1+\sqrt{13}}{2}\right]_q$, its generating function $B^{(1)}(q):=\frac{B(q)-1}{q}$ is such that
\begin{equation}
\label{BroShift1}
q^2B^{(1)}(q)^2 + \left(1-q^2-q^4\right)B^{(1)}(q) = 1+q+q^3.
\end{equation}

(d) The next example is $y_4=\sqrt{5}+2$, sometimes called ``platinum ratio''. The generating function of $[y_4]_q=\left[\sqrt{5}+2\right]_q=q^2\left[\sqrt{5}\right]_q+q+1$ is
$$
P(q)=
\frac{q^5+q^3+q^2+2q-1+\sqrt{(1-q+q^2)(1+q+2q^2+3q^3+6q^4+3q^5+2q^6+q^7+q^8)}}{2q}.
$$
\end{ex}

Further examples can be found in~\cite{LMG}; see Example~4.5.

\section{$q$-irrationals and super $\d$-fractions} \label{HFrSec}

This section contains the material that will be necessary for the proof of our main results.
Namely, we produce various super $\delta$-fraction expansions, in the sense of~\cite{Han}, for the $q$-deformation of the metallic numbers. Some of them will be used in Section~\ref{HankSec} to compute Hankel determinants of the golden and silver ratios.

\subsection{Super $\d$-fractions and H-fractions}

A special class of generalized Jacobi fractions
was introduced and studied by G.-N.~Han in~\cite{Han}.
For every positive integer $\d$, consider the expressions
\begin{equation}
\label{HanFr}
H(q)
=
\cfrac{v_0q^{k_0}}{1+q\,U_1(q)
          - \cfrac{v_1q^{k_0+k_1+\d}}{1+q\,U_2(q)
          -\cfrac{v_2q^{k_1+k_2+\d}}{1+q\,U_3(q)
          - \cfrac{v_3q^{k_2+k_3+\d}}{\ddots
       }}}} 
\end{equation}
where $v_i\neq0$ are constants, $k_i\in\Z_{\geq0}$, and $U_i(q)$ are polynomials 
such that $\deg(U_i)\leq{}k_{i-1}+\d-2$.
These continued fractions were called in~\cite{Han} ``super $\d$-fractions''.
They include the regular C-fractions (for~$\d=1$ and~$k_i\equiv0$)
and the J-fractions (for~$\d=2$ and~$k_i\equiv0$).
One of the main results of~\cite{Han} is that for every $\d\geq1$ any power series can be expanded
as a unique super $\d$-fraction.

In the special case where $\d=2$, the continued fractions~\eqref{HanFr}
were called {\it H-fractions} and applied to computation of Hankel determinants. In particular, 
Theorem 2.1 of~\cite{Han} states the following.
Introduce the notation
\begin{equation}
\label{sEq}
s_n:=\sum_{i=0}^{n-1}k_i+n,
\qquad
\e_n:=\sum_{i=0}^{n-1}\frac{k_i\left(k_i+1\right)}{2},
\qquad \text{for }n\geq 1.
\end{equation}
Then 
\begin{equation}
\label{HanHanEq}
\left\{\begin{array}{rcl}
\Delta_{s_n}\left(H(q)\right)
&=&(-1)^{\e_n}v_0^{s_n}v_1^{s_n-s_1}v_2^{s_n-s_2}\cdots{}v_{n-1}^{s_n-s_{n-1}},  \\[4pt]
\Delta_m\left(H(q)\right)&=&0 \quad\text{if }m\notin\{s_n,\ n\geq 1\}.
\end{array}
\right.
\end{equation}
This theorem is a powerful tool that we will use systematically.
Note  that~\cite{Han} also contains an efficient  algorithm for producing H-fractions
for power series with generating functions satisfying quadratic functional equations.
This algorithm can be applied to our examples, but, for the golden and silver $q$-numbers we give direct simple proofs.
We also refer to~\cite{HanE} for a long history of this statement
that was independently proved in different forms and with different generality 
by several authors; see, e.g.~\cite{Bus}.

The proof of~\eqref{HanHanEq} is based on the following beautiful lemma
 that we will often use directly.

\begin{lem}[Lemma 2.2 of~\cite{Han}]
\label{Han2Lem}
Let $k$ be a nonnegative integer and let $F(q), G(q)$ be two power series such that
$$
F(q)=\frac{q^k}{1+q\,U(q) - q^{k+2}\,G(q)}
$$
where $U(q)$ is a polynomial of $\deg(U)\leq{}k$. Then,
$\D_n(F ) = (-1)^{\frac{k(k+1)}{2}}\,\D_{n-k-1}(G).$
\end{lem}

\subsection{Super $\d$-fractions of metallic $q$-numbers} \label{XXX}

For convenience, we will use the following notation:
\begin{equation}\label{defnq}
\begin{array}{rcl}
\langle{}n\rangle_q
&:=&
q[n]_q+(1+q^n)(1-q)\\[6pt]
&=&
\left\{
\begin{array}{ll}
1+q^2+q^3+\cdots+q^{n-1}+2q^n-q^{n+1} &\text{ if } n\geq 3,\\[4pt]
1+2q^2-q^3&\text{ if } n=2,\\[4pt]
1+q-q^2 &\text{ if } n=1.
\end{array}
\right.
\end{array}
\end{equation}

The following statement is our most general result.
\begin{thm}
\label{FristThm}
(i)
If $y_n$ is a metallic number, i.e. $y_n=\frac{n+\sqrt{n^2+4}}{2}=[n,n,n,\ldots]$ for some integer $n$, then we have the following $1$-periodic expansions:
\begin{equation}\label{GenHFY}
[y_n]_q = [n]_q+
\cfrac{q^{2n}}{\langle{}n\rangle_q
          + \cfrac{q^{2n+1}}{\langle{}n\rangle_q
          +\cfrac{q^{2n+1}}{\ddots}}}
\end{equation} 
and
\begin{equation} \label{GenHF}
\si_q(y_n)=\cfrac{q^{n-1}}{\langle{}n\rangle_q
          + \cfrac{q^{2n+1}}{\langle{}n\rangle_q
          +\cfrac{q^{2n+1}}{\ddots}}}
\end{equation}
where $\si_q(y_n)$ is defined by \eqref{defSiq}.

(ii)
The continued fraction~\eqref{GenHF} is a super $\d$-fraction with $\d=3$.
\end{thm}

\begin{proof}
Part~(i). 
We transform~\eqref{RelY} into a quadratic equation for $\si_q(y_n)$.
Indeed, \eqref{defSiq} reads $\left[y_n\right]_q=[n]_q +q^{n+1}\si_q(y_n)$, and inserting this equality in~\eqref{RelY} leads after some calculation to
\begin{equation*}
q^{n+2}\si_q(y_n)^2+\langle{}n\rangle_q\,\si_q(y_n)-q^{n-1}=0
\end{equation*}
so that
\begin{equation}\label{RelX}
\si_q(y_n)=\frac{q^{n-1}}{\langle{}n\rangle_q+q^{n+2}\si_q(y_n)}
\end{equation}
which is clearly equivalent to~\eqref{GenHF} and thus to~\eqref{GenHFY}.

Part~(ii).
The continued fraction~\eqref{GenHF} fits into the general formula~\eqref{HanFr} with~$\d=3$,
taking~$v_0=1$ and all other coefficients $v_i=-1$, and $k_i=n-1$ for all~$i$.
\end{proof}

\begin{ex}
For~$n=1$ formula~\eqref{GenHF} coincides with~\eqref{GoldHF},
for $n=2$ this is~\eqref{RootTwoHF}.
\end{ex}

\begin{rem}
The continued fractions~\eqref{GenHF} are very simple and $1$-periodic.
They are of type~\eqref{HanFr}, but
unfortunately they are not H-fractions, since~$\d=3$.
Therefore, the methods of~\cite{Han} to calculate the Hankel determinants cannot be applied to them.
Our next goal is to rewrite these super $\d$-fractions as H-fractions, we succeeded to do this in several cases.
\end{rem}

\subsection{H-fractions for the golden ratio}

Let us now rewrite the continued fraction expansions for the series $G(q)$ (see~\eqref{GoEq})
and its shifts,
in such a way that they become H-fractions, i.e.  super $\d$-fractions with~$\d=2$.
This will be useful for the proof of Theorem~\ref{HankelGold}. 

\begin{lem}
\label{ManuFr}
The function $G^{(2)}(q)$ has the following $3$-periodic H-fraction expansion:
\begin{equation}
\label{GoldH}
G^{(2)}(q)
\quad=\quad
\cfrac{1}{1+q
          - \cfrac{q^2}{1+q
          +\cfrac{q^{3}}{1+q-q^2
          +\cfrac{q^3}{ 1+q
          - \cfrac{q^2}{ \ddots}}}}} 
\end{equation}
\end{lem}

\begin{proof}
On the one hand, expression \eqref{GoldH} directly follows from~\eqref{GRBisEq}. On the other hand, 
in  formula~\eqref{HanFr} with $\d=2$, take the
following $3$-periodic sequences of coefficients $k_i,v_i$ and polynomials $U_i$
\begin{equation}\label{kiviUi}
\begin{array}
{rcrrrrrrrl}
k_i &=& 0, & 0, &1, &0,& 0, &1,&0, &\ldots\\
v_i &=&1, & 1, & -1, & -1,& 1, & -1, & -1, &\ldots\\
U_i &= && 1, & 1, & 1-q,& 1, & 1, & 1-q, & \ldots
\end{array}
\end{equation}
where $i=0,1,2,3,\ldots$
With this choice, the H-fraction~\eqref{HanFr} is precisely the continued fraction~\eqref{GoldH}.
\end{proof}

The proofs of the two following statements are similar to that of Lemma~\ref{ManuFr}.

\begin{lem}
One has the following continued fraction expansion
\begin{equation}
\label{GoldHNoShift}
G(q)
\quad=\quad
\cfrac{1}{1
          - \cfrac{q^2}{1+q
          +\cfrac{q^{3}}{1+q-q^2
          + \cfrac{q^3}{ 1+q
          - \cfrac{q^2}{ 1+q
          +\cfrac{q^{3}}{1+q-q^2
          +\cfrac{q^{3}}{1+q
          - \cfrac{q^2}{ 
          \ddots}}}}}}}} 
\end{equation}
which is an H-fraction, $3$-periodic starting from the third numerator.
\end{lem}

\begin{lem}
One has the following $2$-periodic continued fraction expansion
\begin{equation}
\label{GoldHDShift}
G^{(3)}(q)
\quad=\quad
\cfrac{1}{1+2q
          - \cfrac{q^4}{1+q-q^2+2q^3
          -\cfrac{q^{4}}{1+2q
          -\cfrac{q^4}{ 1+q-q^2+2q^3
          - \cfrac{q^4}{ \ddots}}}}} 
\end{equation}
which is also an H-fraction.
\end{lem}

\subsection{H-fractions for the silver ratio}\label{HSilverSec}
Let us do the similar work in the case of our second main example, the silver ratio~$y_2$.
We present H-fractions for the series $S(q),S^{(1)}(q)$ and~$S^{(3)}(q)$ given by~\eqref{NotSilver}. They will be used to prove Theorem~\ref{SilverHanThm}.

\begin{lem}
\label{LemSilH1}
One has the following $8$-periodic H-fraction presentation
\begin{equation}
\label{LongFrac}
S^{(1)}(q)
=
\cfrac{1}{1-
           \cfrac{q^3}{1+2q^2+
          \cfrac{q^5}{1+2q^2-q^3+
          \cfrac{q^5}{ 1+2q^2-
          \cfrac{q^3}{ 1+
          \cfrac{q^2}{1+
          \cfrac{q^2}{1+q+
          \cfrac{q^2}{1+q^2\,S^{(1)}(q)}
          }}}}}}}
\end{equation}
\end{lem}
\begin{proof}
The series $S^{(1)}(q)$ and $S^{(3)}(q)$ are related via the continued fractions
\begin{equation}
\label{Sqrt2FromSiDel}
S^{(1)}(q)
=
\cfrac{1}{1-
           \cfrac{q^3}{1+2q^2+
          \cfrac{q^5}{1+2q^2-q^3+q^4\,S^{(3)}(q)}}}
\end{equation}
and
\begin{equation}
\label{Sqrt2FromSiDelBis}
S^{(3)}(q)=
\cfrac{q}{1+2q^2-
           \cfrac{q^3}{1+
          \cfrac{q^2}{1+
          \cfrac{q^2}{ 1+q+
          \cfrac{q^2}{ 1+q^2\,S^{(1)}(q)}}}}}
\end{equation}
Indeed, using $S^{(1)}(q)=q^2S^{(3)}(q)+1$, formula~\eqref{Sqrt2FromSiDel} is equivalent to
$$
S^{(1)}
=
\cfrac{1}{1-
           \cfrac{q^3}{1+2q^2+
          \cfrac{q^5}{1+q^2-q^3+q^2\,S^{(1)}}}}
$$
which follows from the first equation in~\eqref{ShSRFEq};
similarly for~\eqref{Sqrt2FromSiDelBis}.
Gluing~\eqref{Sqrt2FromSiDel} and~\eqref{Sqrt2FromSiDelBis}, we obtain~\eqref{LongFrac}.

It remains only to check that it has the type \eqref{HanFr} of an H-fraction. This is easily done, by using the following $8$-periodic 
sequences of coefficients $k_i,v_i$ and polynomials $U_i$ 
\begin{equation}\label{kiviUiS}
\begin{array}
{rcrrrrrrrrrl}
k_i &=& 0, & 1, &2, &1,& 0, &0, &0, &0, &\ldots&\\
v_i &=&1, & 1, & -1, & -1,& 1, & -1, & -1, &-1, &\ldots&\\
U_i &=& & 0, & 2q, & 2q-q^2,& 2q, & 0, & 0, &1,& q& \ldots
\end{array}
\end{equation}
where $i=0,1,2,3,\ldots$
\end{proof}

\begin{lem}
\label{LemSilH0}
One has the following relation between the generating function~$S(q)$ 
and its first shift
\begin{equation}
\label{DelFromSqrt2}
S(q)
=
\cfrac{1}{1-q+
           \cfrac{q^2}{1+q+
          \cfrac{q^2}{1+q^2\,S^{(1)}(q)}}}.
\end{equation}
\end{lem}

Note that the concatenation of \eqref{DelFromSqrt2} and \eqref{LongFrac} gives an H-fraction for the $q$-deformation $S(q)$ of the silver ratio $y_2$.

\begin{proof}
This follows from~\eqref{SREq} combined with 
$S(q)=q\,S^{(1)}(q)+1$.
\end{proof}

\section{Hankel determinants of $q$-metallic numbers} \label{HankSec}

In this section, we prove Theorems~\ref{HankelGold},~\ref{RecoverThm} and~\ref{SilverHanThm}.
We then consider more examples of the series~$\left[y_n\right]_q$, and
observe experimentally that several first sequences of their Hankel determinants consist of $-1,0$, and~$1$ only.
Moreover, the number of $\{-1,0,1\}$-Hankel sequences increases, as $n$ grows.

The most fascinating property of these $\{-1,0,1\}$-Hankel sequences is that they satisfy
Somos or Gale-Robinson recurrences.
This property is easily proved for the gold and silver ratio and remains conjectural
for other metallic numbers.

\subsection{Shifted Hankel determinants} \label{ShMH}

We first establish a general formula concerning Hankel determinants of metallic numbers.

\begin{prop}
\label{HanShiftProp}
Let $y_k=[k,k,k,\ldots]$ be a metallic number. We have the following relation between the shifted Hankel determinants:
\begin{equation*}
\D_n^{(k)}\bigl(\left[y_k\right]_q\bigr)=
(-1)^{n+\frac{(k+1)(k-2)}{2}}\D_{n-k-1}^{(k+1)}\bigl(\left[y_k\right]_q\bigr).
\end{equation*}
\end{prop}

\begin{proof}
Let $x_k=y_k-k$ and $\si_q(y_k)=\frac{[x_k]_q}{q}$ as in \eqref{defSiq}. By definition, the power series $[x_k]_q$ and $\si_q(y_k)$ satisfy the relations
\begin{equation*}
[y_k]_q=[k]_q+q^k[x_k]=[k]_q+q^{k+1}\si_q(y_k)
\end{equation*}
hence $\D_n^{(k)}([y_k]_q)=\D_n([x_k]_q)$ and $\D_n^{(k+1)}([y_k]_q)=\D_n(\si_q(y_k))$. On the other hand, from~\eqref{RelX} we have
\begin{equation*}
\left[x_k\right]_q=\frac{q^k}{\langle{}k\rangle_q+q^{k+2}\si_q(y_k)}.
\end{equation*}
where, according to \eqref{defnq}, $\langle{}k\rangle_q$ is of the form $1+u(q)q$ with $u(q)$  a polynomial of degree $k$. 
Thus we can apply Lemma~\ref{Han2Lem} and get that
\begin{equation*}
\D_n(\left[x_k\right]_q)= (-1)^{k(k+1)/2}\D_{n-k-1}(-\si_q(y_k))
=(-1)^{k(k+1)/2}(-1)^{n-k-1}\D_{n-k-1}(\si_q(y_k)),
\end{equation*}
hence the result.
\end{proof}

\subsection{The case of golden ratio. Proof of Theorems~\ref{HankelGold} and~\ref{RecoverThm}} \label{MetalH}

We are ready to prove that the first Hankel determinants of the $q$-deformed golden ratio are
indeed given by~\eqref{HankelGoldEq} and that they characterize the series~\eqref{GoEq}. 

\textit{Proof of Theorem~\ref{HankelGold}}.
Until the end of the proof the Hankel determinants~$\Delta_n^{(\ell)}\left(G\right)$ will be simply denoted by $\Delta_n^{(\ell)}$ when there is no ambiguity.

It is clear that it suffices to prove Part (1) of the theorem, and here is our  strategy: 

(a) first, we shall prove that the knowledge of any of the first three sequences of determinants $\Delta=\D^{(0)},\Delta^{(1)},\Delta^{(2)}$ entails the knowledge of the two others, because there are very simple relations between them. 

Then it will be sufficient:

(b) to calculate $\Delta^{(2)}$, and

(c) to calculate $\Delta^{(3)}$.

Before we do so, a first obvious but crucial observation is that the three shifted determinants $\Delta_n^{(\ell)}$, where $\ell=1,2,3$, correspond to the non-shifted Hankel determinant $\D_n$ of a shifted variant of $G(q)$ or, equivalently, of a power series represented by one of the generating functions $G^{(\ell)}$, $\ell=1,2,3$ introduced in~\eqref{NotGold}. To make it clear:
\begin{equation*}
\D_n^{(\ell)}=\D_n(G^{(\ell)}),\quad \ell=0,1,2,3.
\end{equation*}

Now we start our proof.

(a) 
Applying Proposition~\ref{HanShiftProp} to the case of the golden ratio $\varphi=y_1$,
one obtains the relation
\begin{equation*}
\D_n(G^{(1)})=(-1)^{n-1}\D_{n-2}(G^{(2)})
\end{equation*}
or equivalently
\begin{equation*}
\D_n^{(1)}=(-1)^{n-1}\D_{n-2}^{(2)}.
\end{equation*}
Thus the second and third sequences in \eqref{HankelGoldEq} can be deduced the one from the other.
\\
Similarly, because of \eqref{TrueGoldCF} and \eqref{NotGold} we have
\begin{equation*}
G(q)=\frac{1}{1-q^2F_1(q)},\quad\text{with}\quad
F_1(q)=\frac{1}{1+q G(q)}=\frac{1}{1+q+q^2 G^{(1)}(q)}.
\end{equation*}
Applying Han's Lemma~\ref{Han2Lem} we thus obtain the relation 
$$\D_n=\D_n(G)=(-1)^n\D_n(G^{(1)})=(-1)^n \D_{n-2}^{(1)}$$
which connects the second sequence in \eqref{HankelGoldEq} with the first one.

(b) Now we determine explicitly the shifted Hankel determinants~$\Delta_n^{(2)}$. Since  the continued fraction~\eqref{GoldH} is an H-fraction by Lemma~\ref{ManuFr}, we can calculate the Hankel determinants~$\Delta_n^{(2)}$ via  formula~\eqref{HanHanEq}. 
The parameters $v_i,s_i,\e_i$  contributing in this formula can be easily deduced from~\eqref{kiviUi} and~\eqref{sEq}:
$$
\begin{array}
{rcrrrrrrrrrrl}
v_i &=&1, & 1, & -1, & -1,& 1, & -1, & -1, & -1, & -1, &\ldots\\
s_i &=& & 1, & 2, & 4,& 5, & 6, & 8, & 9, &10,&12,&\ldots\\
\e_i &=& & 0, & 0, & 1,& 1, & 1, & 2, & 2, &2,& 3,&\ldots
\end{array}
$$
where $i=0,1,2,3,\ldots$
The sequence $(s_i)$  misses the values $3+4i=3,7,11,\ldots$  which means that the Hankel determinants~$\Delta^{(2)}_{3+4i}$ vanish. On the other hand, the non-zero values of the Hankel determinants
of the sequence $\si_q(\varphi)$ are as follows:
$$
\Delta^{(2)}_0=1,\quad
\Delta^{(2)}_1=1,\quad
\Delta^{(2)}_2=1,\quad
\Delta^{(2)}_4=-1,\quad
\Delta^{(2)}_5=-1,\quad
\Delta^{(2)}_6=-1,\quad
\Delta^{(2)}_8=1,\quad
\ldots
$$
In other words we have proved that the sequence $\Delta^{(2)}$ is the one given in
the third row of~\eqref{HankelGoldEq}. Because of (a) this proves also that $\D$ and $\D^{(1)}$ are as in~\eqref{HankelGoldEq}.

(c)
The remaining case of $\Delta^{(3)}=\D(G^{(3)})$ in~\eqref{HankelGoldEq} is treated by applying formula~\eqref{HanHanEq} to the H-fraction~\eqref{GoldHDShift}.
The parameters of this H-fraction are $k_{2m}=0,\,k_{2m+1}=2$ for all~$m\geq0$, and thus
$$
s_i=1,4,5,8,9,12,13\ldots,
\qquad
i=1,2,3,\ldots
$$
are the indices of non-zero determinants. We  obtain this way the fourth row  in~\eqref{HankelGoldEq}.

Theorem~\ref{HankelGold} is proved.

\begin{rem}
\label{RemHankelGold}
An alternative proof of (b) above would be the following. According to~\eqref{GoldH}, one can write
\begin{equation*}
G^{(2)}(q)=\frac{1}{1+q-q^2F_1(q)},\quad 
F_1(q)=\frac{1}{1+q+q^2 F_2(q)},\quad 
F_2(q)=\frac{q}{1+q-q^2+q^3 G^{(2)}(q)}.
\end{equation*}
Applying Han's Lemma~\ref{Han2Lem} we thus obtain the relations 
$$\D_n(G^{(2)})=\D_{n-1}(F_1),\quad \D_n(F_1)=(-1)^{n-1}\D_{n-1}(F_2),
\quad \D_n(F_2)=(-1)^{n-1}\D_{n-2}(G^{(2)})$$
which imply the $4$-antiperiodicity: $\D_n^{(2)}=-\D_{n-4}^{(2)}$. Thus it suffices to calculate the first four determinants $\D_n^{(2)}$ with $n=0,1,2,3$ to get the complete sequence. 

Similarly, using \eqref{GoldHDShift} we easily prove the $4$-antiperiodicity of the $\D^{(3)}$ sequence and this gives another proof of (c).

Finally, let us mention that, instead of proving the validity of our formula~\eqref{HankelGoldEq} for the sequence of shifted determinants $\D^{(2)}$ as we did in (b), we could have looked instead at the sequence of non-shifted determinants $\D$, since (a) above shows that they are equivalent. The proof is quite the same and consists in applying formula~\eqref{HanHanEq} to the H-fraction~\eqref{GoldHNoShift}.
\end{rem}

\textit{Proof of Theorem~\ref{RecoverThm}}.
Let us now prove that the series $G(q)$ 
is characterized by the Hankel determinants~\eqref{HankelGoldEq}.
We proceed by induction: assume that the first $k$ coefficients, $G_0,G_1,\ldots, G_{k-1}$ of the series
$$
G(q)=\sum_{i=0}^\infty G_iq^i
$$
(see \eqref{GoEq}) are determined by~\eqref{HankelGoldEq}.
We need to prove that $G_k$ is also determined by these determinants.

Among the Hankel determinants~\eqref{HankelGoldEq}, consider those with $G_k$ in the lower right entry:
\begin{equation}
\label{HanK}
\Delta^{(\ell)}_n(G)=
\left|
\begin{array}{rl}
\Delta^{(\ell)}_{n-1}&\vdots
\\[2pt]
\cdots&G_k
\end{array}
\right|,
\end{equation}
where $\ell=0,1,2,3$ and $\ell+2n-2=k$ (see~\eqref{HankDet}).
For every $G_k$ there are exactly two such Hankel determinants:
if $k$ is even, these are $\Delta_{\frac{k}{2}+1}$ and $\Delta^{(2)}_{\frac{k}{2}}$;
if $k$ is odd, these are $\Delta^{(1)}_{\frac{k+1}{2}}$ and $\Delta^{(3)}_{\frac{k-1}{2}}$.

It is easy to see that, according to Theorem~\ref{HankelGold} (see~\eqref{HankelGoldEq}), 
at least one of the determinants $\Delta^{(\ell)}_{n-1}$ in~\eqref{HanK}
 is different from~$0$.
Indeed, for every~$k$, either $\Delta_{\frac{k}{2}+1}\neq0$, or $\Delta^{(2)}_{\frac{k}{2}}\neq0$,
and  either $\Delta^{(1)}_{\frac{k+1}{2}}\neq0$, or $\Delta^{(3)}_{\frac{k-1}{2}}\neq0$.
Since the row and the column containing~$G_k$ in~\eqref{HanK} consist of coefficients
$G_i$ with $i<k$, which are known by induction hypothesis, we conclude that
the value of~\eqref{HanK} determines $G_k$, provided $\Delta^{(\ell)}_{n-1}$ is different from~$0$.

Theorem~\ref{RecoverThm} is proved.

\subsection{The case of silver ratio. Proof of Theorem~\ref{SilverHanThm}}\label{ProofSilerSec}

For short, we will use the notation $\Delta_n^{(\ell)}$ instead of $\Delta_n^{(\ell)}\left(S\right)$ until the end of the proof, so that
\begin{equation*}
\D_n^{(\ell)}=\D_n(S^{(\ell)}),\quad \ell=0,1,2,3,4,
\end{equation*}
with $S^{(\ell)}$ defined in~\eqref{NotSilver}. Let us prove Part (1) of the theorem, which immediately implies Part (2).

(a)
Consider first the shifted series~$S^{(1)}(q)$.
Applying~\eqref{HanHanEq} to the continued fraction of
Lemma~\ref{LemSilH1}, we obtain the second formula in~\eqref{DelHan}.
Indeed, the parameters $v_i,s_i,\e_i$  contributing in this formula can be easily deduced from~\eqref{kiviUiS} and~\eqref{sEq}:
$$
\begin{array}
{rcrrrrrrrrrl}
v_i &=&1, & 1, & -1, & -1,& 1, & -1, & -1, & -1, & -1, &\ldots\\
s_i &=& & 1, & 3, & 6,& 8, & 9, & 10, & 11, &12,&\ldots\\
\e_i &=& & 0, & 1, & 4,& 5, & 5, & 5, & 5, &5,&\ldots
\end{array}
$$
where $i=0,1,2,3,\ldots$
The Hankel sequence $\D^{(1)}_n$ is then given by~\eqref{HanHanEq}.

(b)
Now, Lemma~\ref{LemSilH0} allows us to calculate the Hankel determinants~$\D_n$.
Indeed, using auxillary functions~$F_1(q)$ and~$F_2(q)$, formula~\eqref{DelFromSqrt2} reads
$$
S(q)=\frac{1}{1-q+q^2\,F_1(q)},
\qquad
F_1(q)=\frac{1}{1+q+q^2\,F_2(q)},
\qquad
F_2(q)=\frac{1}{1+q^2\,S^{(1)}(q)}.
$$
Applying then Han's Lemma~\ref{Han2Lem}, we have
$$
\D_n=(-1)^{n-1}\D_{n-1}(F_1),
\qquad
\D_n(F_1)=(-1)^{n-1}\D_{n-1}(F_2),
\qquad
\D_n(F_2)=(-1)^{n-1}\D^{(1)}_{n-1}.
$$
Hence, we conclude $\D_n=(-1)^n\D^{(1)}_{n-3}$, in accordance with~\eqref{DelHan}.

(c)
Applying~\eqref{Sqrt2FromSiDel}, we are able to calculate the determinants~$\D^{(3)}_n$.
Once again, using auxillary functions~$F_1(q)$ and~$F_2(q)$,
formula~\eqref{Sqrt2FromSiDel} then gives
$$
S^{(1)}(q)=\frac{1}{1-q^2\,F_1(q)},
\qquad
F_1(q)=\frac{q}{1+2q^2+q^3\,F_2(q)},
\qquad
F_2(q)=\frac{q^2}{1+2q^2-q^3+q^4\,S^{(3)}(q)}.
$$
From Han's Lemma, we have
$$
\D^{(1)}_n=\D_n(F_1),
\qquad
\D_n(F_1)=(-1)^{n-1}\D_{n-2}(F_2),
\qquad
\D_n(F_2)=(-1)^{n}\D^{(3)}_{n-3},
$$
and therefore $\D^{(1)}_n=-\D^{(3)}_{n-6}$ confirming the fourth row of~\eqref{DelHan}.

(d)
Finally, we apply Proposition~\ref{HanShiftProp} to deduce the row~$\D^{(2)}_n$ in~\eqref{DelHan} from~$\D^{(3)}_n$:
\begin{equation*}
\D_n^{(2)}=(-1)^{n}\D^{(3)}_{n-3}.
\end{equation*}

Theorem~\ref{SilverHanThm} is proved.

\begin{rem}
\label{RemLem}
As for Theorem~\ref{HankelGold}, let us mention that there are alternative proofs of Theorem~\ref{SilverHanThm}. 
For instance,~\ref{Sqrt2FromSiDel} and~\ref{Sqrt2FromSiDelBis} connect the determinants $\D^{(3)}_n$ with $\D^{(1)}_n$.
Indeed, using the same method as in (b) and (c), we have
$\D^{(3)}_n=-\D^{(1)}_{n-6}$.
Combining this with the conclusion (c), one proves the $12$-periodicity of
the Hankel sequence $\D^{(3)}_n$, and thus of $\D_n,\D^{(1)}_n$, and $\D^{(2)}_n$.
Then to prove Theorem~\ref{SilverHanThm} it suffices to calculate the first $11$ values of (one of) the sequences.
\end{rem}

\subsection{The case of $y_3$ and $y_4$: more rows of $-1,0,1$}\label{Metal3Sec}
Consider the next example (after~$\d$)  of metallic number $y_3=\frac{3+\sqrt{13}}{2}$.
Computer assisted proofs, in particular, application of Algorithms
of~\cite{Han}, show that the phenomenons already observed for~$\varphi$ and~$\d$
become even more amazing:
the number of Hankel sequences consisting of~$-1,0,1$ increases.

\begin{fact}
The first five sequences of Hankel determinants~$\Delta^{(\ell)}_n$ associated with
the series~$\left[\frac{3+\sqrt{13}}{2}\right]_q$ generated by the function~\eqref{BronzeGF},
consist of~$-1,0$ and~$1$ only.
These sequences are $24$-antiperiodic
$$
\Delta^{(\ell)}_{n+24}=-\Delta^{(\ell)}_n,
\qquad
\ell=0,1,2,3,4,
$$
with the following (anti)periods
\begin{equation}
\label{Hank3Eq}
\begin{array}{rcl}
\Delta_n&=&1,1,0,-1,-1,1,1,0,-1,-1,0,0,1,0,0,0,1,0,0,-1,-1,0,1,1,\ldots\\[4pt]
\Delta^{(1)}_n&=&1,1,-1,0,1,-1,0,0,-1,0,0,0,-1,0,0,-1,1,0,-1,1,1,-1,0,1,\ldots\\[4pt] 
\Delta^{(2)}_n&=&1,1,0,0,-1,0,0,0,-1,0,0,1,1,0,-1,-1,1,1,0,-1,-1,1,1,0,\ldots\\[4pt]
\Delta^{(3)}_n&=&1,0,0,0,1,0,0,1,-1,0,1,-1,-1,1,0,-1,1,1,-1,0,1,-1,0,0,\ldots\\[4pt]
\Delta^{(4)}_n&=&1,0,0,-1,-1,0,1,1,-1,-1,0,1,1,-1,-1,0,1,1,0,0,-1,0,0,0,\ldots
\end{array}
\end{equation}
\end{fact}

Note that the algorithms of~\cite{Han} provide a ``computer assisted proof'' of the above statement.

\begin{rem}
The above sequences $\Delta^{(0)}_n,\ldots,\Delta^{(4)}_n$ have multiple symmetries.
In particular, they satisfy
\begin{equation}
\label{ShiftRelationsBronse}
\Delta_n\; =\; 
(-1)^{n-1}\Delta^{(1)}_{n-4}\; =\; 
-\Delta^{(2)}_{n-8}\; =\; 
(-1)^n\Delta^{(3)}_{n-12}\; =\; 
\Delta^{(4)}_{n-16}.
\end{equation}
\end{rem}

Some of the above relations are easy to prove.
For instance, the last equality in~\eqref{ShiftRelationsBronse} is a corollary of Proposition~\ref{HanShiftProp},
while the second equality follows from the following.

\begin{lem}
\label{FlightLem}
One has
$$
B^{(1)}(q)
\;=\;
\cfrac{1}{1-q+
           \cfrac{q^2}{1+q+
          \cfrac{q^3}{1+q^2\,B^{(1)}(q)}}}
\;=\;
\cfrac{1}{1-q+
           \cfrac{q^2}{1+q+
          \cfrac{q^3}{1+q+q^3\,B^{(2)}(q)}}}.
$$
\end{lem}

\begin{proof}
This readily follows from~\eqref{BroShift1}.
\end{proof}

Using Lemma~\ref{Han2Lem} we get the following.

\begin{cor}
\label{SecThird}
The second and the third rows are related via
$\D^{(1)}_n=(-1)^n\D^{(2)}_{n-4}$.
\end{cor}

The other two relations in~\eqref{ShiftRelationsBronse} require more sophisticated computations, 
and we do not elaborate the details here.

The next Hankel sequences have entries different from~$-1,0,1$, but certain symmetry persists.
For example we have a sequence which is again $24$-antiperiodic with the following (anti)period
\begin{eqnarray*}
\Delta^{(5)}_n&=&1,0,-1,2,1,-1,1,0,-1,0,0,-1,0,1,-1,1,2,-1,0,1,0,0,0,0,\ldots
\end{eqnarray*}

Let us end this section with the case of $\sqrt{5}+2=y_4$.
\begin{fact}
The first six sequences of Hankel determinants of $\left[\sqrt{5}+2\right]_q$
consist of $-1,0,1$ only, and they are $40$-periodic.
The sequences start like this
\begin{equation}\label{HankelPlatine}
\begin{array}{rcrrrrrrrrrrrrrr}
\D_n
&=&
1,&\;1,&0,&0,&\;1,&\;1,&-1,&0,&\;1,&0,&-1,&-1,&\;1,&\ldots\\[4pt]
\D^{(1)}_n
&=&
1,&1,&0,&-1,&0,&1,&-1,&-1,&0,&0,&-1,&-1,&0,&\ldots\\[4pt]
\D^{(2)}_n
&=&
1,&1,&-1,&0,&0,&1,&-1,&0,&0,&0,&-1,&0,&0,&\ldots\\[4pt]
\D^{(3)}_n
&=&
1,&1,&0,&0,&0,&1,&0,&0,&0,&0,&-1,&0,&0,&\ldots\\[4pt]
\D^{(4)}_n
&=&
1,&0,&0,&0,&0,&1,&0,&0,&0,&1,&-1,&0,&0,&\ldots\\[4pt]
\D^{(5)}_n
&=&
1,&0,&0,&0,&1,&1,&0,&0,&1,&1,&-1,&0,&1,&\ldots
\end{array}
\end{equation}
The following relations hold between the rows:
\begin{equation*}
\D_n^{(\ell)}=(-1)^{n-1}\D_{n-5}^{(\ell+1)}
\end{equation*}
for  $\ell=0,1,2,3,4$.
\end{fact}

The algorithms of~\cite{Han} also provide  ``computer assisted proof'' of this statement.

\subsection{The Somos-4, Somos-6 and Gale-Robinson recurrences}\label{GRSec}
(a) Consider first the case of the golden ratio.
We already know the first rows of Hankel determinants of~$G(q)$
(see Theorem~\ref{HankelGold}), and it is an easy task to check that the first three rows 
$\D_n(G),\D^{(1)}_n(G)$, and~$\D^{(2)}_n(G)$,
satisfy the Somos-4 recurrence~\eqref{Somos4}.

It is interesting to notice that in every $8$-periodic row satisfying this recurrence
the $4$ first (``initial'') values determine the $4$ remaining.
For instance, in the first row
$$
\D_n(G)\;=\;\mathbf{1,1,1,0,}-1,-1,-1,0, 1,1,1,\ldots
$$
where $n=0,1,2,\ldots$
the values $1,1,1,0$ determine the rest.
Indeed, although the recurrence equation $\D_7\D_3=\D_6\D_4-\D_5^2$ 
does not allow to recover $\D_7$, since $\D_3=0$,
the next equation $\D_8\D_4=\D_7\D_5-\D_6^2$ determines $\D_7$,
since $\D_8$ is known by periodicity.

(b)
In the case of the Silver ratio, the first four rows satisfy the Somos-6 recurrence~\eqref{Somos6}: similarly to the case of the golden ratio, this is an easy consequence of Theorem~\ref{SilverHanThm}.
Moreover, one can check that this Somos-6 recurrence, together with the $12$-periodicity assumption,
allows one to recover the full sequence of coefficients from the six first terms.

Note also that the quadratic functional equations~\eqref{ShSRFEq} fit into the class of
sequences satisfying Eq.~(4.22) of~\cite{Hon2}, so that this observation can be understood as a part of
Theorem 5.5 of this reference.

(c)
The five sequences of Hankel determinants~\eqref{Hank3Eq}
of the bronze number~$y_3=\frac{3+\sqrt{13}}{2}$ satisfy the Gale-Robinson recurrence
$$
\D_{n+8}\D_n = \D_{n+7}\D_{n+1} - \D_{n+4}^2,
$$
in accordance with Conjecture~\ref{MainCon}.

(d) For the metallic number $y_4=\sqrt{5}+2$, the six sequences of Hankel determinants~\eqref{HankelPlatine} satisfy the Gale-Robinson recurrence
\begin{equation}
\label{Somos10}
\D_{n+10}\D_n = \D_{n+9}\D_{n+1} - \D_{n+5}^2.
\end{equation}
As before, recurrence~\eqref{Somos10} allows to recover the Hankel sequences
from the $10$ initial values, provided we know that the sequences are $40$-periodic.

\section{Miscellaneous examples of continued fractions for $q$-numbers} \label{EverySec}

In this section we give several examples of continued fractions for the series
representing $q$-deformed rationals and $q$-deformed  irrationals.
Some of them are quite elegant and allow one to hope for a more general theory
which is out of reach so far.
 
\subsection{Another J-fraction for the silver ratio}\label{SRFRSec}
Let us consider the series
$1-S^{(2)}(q)$ (see~\eqref{NotSilver}) that has essentially the same sequence of coefficients than $S(q)$.
It turns out that this series has a rather nice $2$-periodic J-type continued fraction expansion.

\begin{prop}
\label{SilverJProp}
One has the following $2$-periodic continued fraction
$$
1-S^{(2)}(q)
\quad=\quad
\cfrac{1}{1
          + \cfrac{q^2}{1+q^2
          -\cfrac{q^{3}}{1
          + \cfrac{q^2}{ 1+q^2
          - \cfrac{q^3}{ \ddots}}}}} 
$$
\end{prop}

\begin{proof}
This is an immediate corollary of~\eqref{ShSRFEq}.
\end{proof}

\subsection{A Stieltjes type formula for $S^{(2)}(q)$}
We give here the result of a straightforward computation.
The formula below is rather complicated and does not leave a hope for general results in this direction.
For a symmetry reason, we multiply $S^{(2)}(q)$ by $q^3$.

\begin{prop}
The following $7$-glide-periodic C-fraction represents the (shifted) $[\sqrt{2}]_q$
$$
q^3\,S^{(2)}(q)
=
\cfrac{q^5}{1+
           \cfrac{2q^2}{1+
          \cfrac{\frac{q}{2}}{1-
          \cfrac{\frac{q}{2}}{ 1+
          \cfrac{2q}{ 1-
          \cfrac{2q}{1+
          \cfrac{\frac{q}{2}}{\ddots}
          }}}}}}
$$
\end{prop}

\noindent
By ``$7$-glide-periodic'' we mean that the $7$ consecutive numerators of the above continued fraction repeat with inverse order:
$\frac{q}{2},-2q,2q,-\frac{q}{2},\frac{q}{2},2q^2,q^5,\ldots$
Surprisingly, this continued fraction is $13$-periodic.

\subsection{C-fractions for the $q$-integers}
Let us consider the $q$-deformations of integers $[n]_q=1+q+\cdots+q^{n-1}$, as in~\eqref{EGEq},
and of their reciprocals~$\left[\frac{1}{n}\right]_q$.

\begin{prop}
\label{IntegProp}
(i)
The C-fraction expression representing a $q$-integer is as follows
\begin{equation*}
[n]_q=
\cfrac{1}{1-
           \cfrac{q}{1+
          \cfrac{q^{n-1}}{1+
          \cfrac{q}{ 1-
          \cfrac{q}{ 1+
          \cfrac{q^{n-3}}{1+
          \cfrac{q}{\ddots}}}}}}}
\end{equation*}
If $n=2m+1$ it is of length $3m$, and if $n=2m$ it is of length $3m-1$.

(ii)
For the reciprocal $q$-number $\left[\frac{1}{n}\right]_q$, one has
$$
\left[\frac{1}{n}\right]_q=
\cfrac{q^n}{1+
           \cfrac{q}{1-
          \cfrac{q}{1+
          \cfrac{q^{n-2}}{1+
          \cfrac{q}{1-
          \cfrac{q}{1+
          \cfrac{q^{n-4}}{\ddots}
          }}}}}}
$$
\end{prop}

\begin{proof}
Use \eqref{EGEq} for Part (i) and the formula $\left[\frac{1}{n}\right]_q=\frac{1}{[n]_{q^{-1}}}$ (see Eq. (2.8) in \cite{LMG}) for Part~(ii).
\end{proof}

\begin{ex}
One has for instance
$$
[3]_q=
\cfrac{1}{1-
           \cfrac{q}{1+
          \cfrac{q^{2}}{1+q
          }}}
,
\qquad
[4]_q=
\cfrac{1}{1-
           \cfrac{q}{1+
          \cfrac{q^{3}}{1+
          \cfrac{q}{ 1-
          \cfrac{q}{ 1+q
         }}}}},
\qquad
[5]_q=
\cfrac{1}{1-
           \cfrac{q}{1+
          \cfrac{q^{4}}{1+
          \cfrac{q}{ 1-
          \cfrac{q}{ 1+
          \cfrac{q^{2}}{1+q
          }}}}}}
$$
and for the reciprocal numbers
$$
\left[\frac{1}{3}\right]_q=
\cfrac{q^2}{1+
           \cfrac{q}{1-
          \cfrac{q}{1+q
          }}},
\qquad
\left[\frac{1}{4}\right]_q=
\cfrac{q^3}{1+
           \cfrac{q}{1-
          \cfrac{q}{1+
          \cfrac{q^2}{1+q
          }}}},
\qquad
\left[\frac{1}{5}\right]_q=
\cfrac{q^4}{1+
           \cfrac{q}{1-
          \cfrac{q}{1+
          \cfrac{q^3}{1+
          \cfrac{q}{1-
          \cfrac{q}{1+q
          }}}}}}
$$
\end{ex}

\subsection{C-fractions for $q$-rationals}

The simplest examples for $q$-rationals are as follows.
$$
\left[\frac{2}{5}\right]_q=
\cfrac{q^2}{1+
           \cfrac{2q^2}{1+
          \cfrac{\frac{q}{2}}{1+\frac{q}{2}
          }}},
\qquad\qquad
\left[\frac{3}{5}\right]_q=
\cfrac{q}{1+
           \cfrac{q}{1+
          \cfrac{q}{1+q^2
          }}}
$$
Here we have used the expansions
$$\frac 2{5}=[0,2,2],\qquad \frac{3}{5}=[0,1,1,2],$$
and applied definition \eqref{QxEq2}.

\bigbreak \noindent
{\bf Acknowledgements}.
We are grateful to Guoniu Han, Andrew Hone, Sophie Morier-Genoud, and Michael Somos  for enlightening discussions.
We are also happy to thank Sergei Tabachnikov, Aleksey Ustinov, and Alexander Veselov for helpful comments.

\end{document}